\documentclass{amsart}
\usepackage{latexsym,amssymb,amsthm,amsmath,verbatim}

\begin{document}

\newtheorem{theorem}{Theorem}[section]
\newtheorem{proposition}[theorem]{Proposition}
\newtheorem{definition}[theorem]{Definition}
\newtheorem{corollary}[theorem]{Corollary}
\newtheorem{lemma}[theorem]{Lemma}
\newtheorem{question}[theorem]{Question}
\newtheorem{example}[theorem]{Example}
\newcommand{\Q}{\mathbb{Q}}
\newcommand{\mA}{\mathcal{A}}
\newcommand{\mF}{\mathcal{F}}
\newcommand{\mV}{\mathcal{V}}
\newcommand{\mU}{\mathcal{U}}
\newcommand{\mW}{\mathcal{W}}
\newcommand{\mB}{\mathcal{B}}

\newcommand{\mO}{\mathcal{O}}
\newcommand{\mC}{\mathcal{C}}

\newcommand{\C}{\mathrm{C}}
\newcommand{\D}{\mathrm{D}}
\newcommand{\OD}{\mathrm{OD}}
\newcommand{\Do}{\D_\mathrm{o}}
\newcommand{\sone}{\mathsf{S}_1}
\newcommand{\gone}{\mathsf{G}_1}
\newcommand{\sfin}{\mathsf{S}_\mathrm{fin}}

\newcommand{\gfin}{\mathsf{G}_\mathrm{fin}}

\newcommand{\Em}{\longrightarrow}
\newcommand{\w}{{\omega}}

\title{When is a space Menger at infinity?}

\author[L. F. Aurichi]{Leandro F. Aurichi$^1$}
\thanks{$^1$ Partially supported by FAPESP (2013/05469-7) and by GNSAGA}
\address{Instituto de Ci\^encias Matem\'aticas e de Computa\c c\~ao (ICMC-USP),
Universidade de S\~ao Paulo, S\~ao Carlos, Brazil}

\author[A. Bella]{Angelo Bella$^2$}
\address{Department of Mathematics, University of Catania, Catania, Italy}
\maketitle

\begin{abstract}We try to characterize those Tychonoff spaces $X$
such that $\beta X\setminus X$ has the Menger property.
\end{abstract}
\section {Introduction}
A space $X$ is Menger (or has the Menger property) if for any
sequence of open coverings $\{\mathcal U_n:n<\w\}$ one may pick
finite sets  $\mathcal V_n\subseteq \mathcal U_n$ in such a way
that
$\bigcup\{\mathcal V_n:n<\w\}$ is a covering. This equivals to
say that $X$ satisfies the selection principle $\sfin(\mO,\mO)$.
It is easy to see the following chain of implications:

$${\text{$\sigma$-compact}}  \ \longrightarrow\  {\text{Menger}}
\
\longrightarrow\   {\text{Lindel\"of}}$$

An important result of Hurewicz \cite{hurewicz}   states that  a
space $X$ is Menger if and only if player 1 does not have a
winning strategy in the associated game $\gfin(\mO,\mO)$ played
on $X$. This
highlights  the  game-theoretic nature of the Menger property,
see  \cite{ss} for more.

Henriksen and Isbell
(\cite{henriksen}) proposed the following:

\begin{definition}
 A  Tychonoff space $X$ is Lindel\" of  at infinity if $\beta X
\setminus X$ is  Lindel\" of.
\end{definition}
They discovered a very elegant duality in the following:
\begin{proposition} \cite{henriksen}  \label{henriksen} A
Tychonoff space is
Lindel\"
of at infinity if and only if it is of countable type.
\end{proposition}
A space $X$ is of countable type provided that every compact set
can be included in a compact set of countable character in $X$.

A much easier and well-known fact is:
\begin{proposition} A Tychonoff space is \v Cech-complete if
and only if it is $\sigma$-compact at infinity. \end{proposition}

These two propositions  suggest the following:  
\begin{question} \label {Q1} When is  a Tychonoff space Menger at
infinity? \end{question}

Before beginning  our discussion here, it is useful to note these
well known facts:

\begin{proposition}
  The Menger property is invariant by perfect maps.
\end{proposition}

\begin{corollary}
  $X$ is Menger at infinity if, and only if, for any $Y$
compactification of $X$, $Y \setminus X$ is Menger.
\end{corollary}

Fremlin and Miller \cite{fremlinmiller} proved the existence of a
Menger subspace  $X$  of the  unit interval $[0,1]$ which is not
$\sigma$-compact. The space $X$ can be  taken nowhere locally
compact and so $Y=[0,1]\setminus X$ is dense in $[0,1]$.  Since
the Menger property is invariant under perfect mappings, we see
that $\beta Y\setminus Y$ is still Menger.
Therefore,  a space can be Menger at
infinity and not $\sigma$-compact at infinity. Another example of
this kind,  stronger but not second countable,  is  Example
\ref{ex} in the last section.

On the other hand,
the irrational
line shows that a space can be Lindel\" of at infinity and not
Menger at infinity.

  Consequently, the property $\mathcal M$ characterizing
a space to be Menger at infinity strictly lies between
countable type and \v Cech-complete.

Of course,  taking into account the formal definition of the
Menger  property, we cannot expect to have an answer to Question
\ref{Q1}  as elegant as  Henriksen-Isbell's result.

\section {A characterization}
\begin{definition}
  Let $K \subset X$. We say that a family $\mF$ is a {\bf closed
net at} $K$ if each $F \in \mF$ is a closed set such that $K
\subset F$ and for every open $A$ such that $K \subset A$, there
is an $F \in \mF$ such that $F \subset A$.
\end{definition}

\begin{lemma}\label{interseccao funciona}
  Let $X$ be a  $T_1$ space. If $(F_n)_{n \in \w}$ is a closed
net at $K$, for $K \subset X$ compact, then $K = \bigcap_{n \in
\w} F_n$.
\end{lemma}

\begin{proof}
  Simply note that for each $x \notin K$, there is an open set
$V$ such that $K \subset V$ and $x \notin V$.
\end{proof}

\begin{lemma}\label{nets sobem}
  Let $Y$ be a regular space and let $X$ be a dense subspace of
$Y$. Let $K \subset X$ be a compact subset. If $(F_n)_{n \in \w}$
is a closed net at $K$ in $X$, then $(\overline{F_n}^Y)_{n \in
\w}$ is a closed net at $K$ in $Y$.
\end{lemma}

\begin{proof}
  In the following, all the closures are taken in $Y$. Let $A$ be
an open set in $Y$ such that $K \subset A$. By the compactness of
$K$ and the regularity of $Y$, there is an open set $B$ such that
$K \subset B \subset \overline B \subset A$. Thus, there is an $n
\in \w$ such that $K \subset F_n \subset B \cap X$. Note that $K
\subset \overline F_n \subset \overline B \subset A$.
\end{proof}

\begin{lemma}\label{decrescente eh net}
  Let $X$ be a compact Hausdorff space. If $K = \bigcap_{n \in
\w} F_n$, where $(F_n)_{n \in \w}$ is a decreasing sequence of
closed sets, then $(F_n)_{n \in \w}$ is a closed net at $K$.
\end{lemma}

\begin{proof}
   If not,  then there is an open set $V$ such that $K
\subset V$ and, for every $n \in \w$, $F_n \setminus V \neq
\emptyset$. By compactness, there is an $x \in \bigcap_{n \in \w}
F_n \setminus V = K \setminus V$. Contradiction with the fact
that $K \subset V$.
\end{proof}

\begin{theorem}\label{Menger at infinity}
  Let $X$ be a  Tychonoff space. $X$ is Menger
at infinity if, and only if, $X$ is of countable type and for
every sequence $(K_n)_{n \in
\w}$ of compact subsets of $X$, if $(F_p^n)_{p \in \omega}$ is a
decreasing closed net at $K_n$ for each $n$, then there is an $f:
\w \Em \w$ such that $K = \bigcap_{n \in \w} F_{f(n)}^n$ is
compact and $(\bigcap_{k \leq n} F_{f(k)}^k)_{n \in \w}$ is a
closed net for $K$.
\end{theorem}

\begin{proof}
  In the following, every closure is taken in $\beta X$.

Suppose that $X$ is Menger at infinity. By \ref{henriksen} $X$
is of countable type.
 Let $(F_p^n)_{p, n \in
\w}$ be as in the statement. Note that, by Lemma \ref{nets sobem}
and Lemma \ref{interseccao funciona}, $\bigcap_{p \in \w} F_p^n =
\bigcap_{p \in \w} \overline{F_p^n}$ for each $n \in \w$. Thus,
for each $n \in \w$, $(V_p^n)_{p \in \w}$, where $V_p^n = \beta X
\setminus \overline{F_p^n}$, is an increasing covering for $\beta
X \setminus X$. Since $\beta X \setminus X$ is Menger, there is
an $f: \w \Em \w$ such that $\beta X \setminus X \subset
\bigcup_{n \in \w}  V_{f(n)}^n$. Note that $K = \bigcap_{n \in
\w} \overline{F_{f(n)}^n}$ is compact and it is a subset of $X$.
By Lemma \ref{decrescente eh net}, $(\bigcap_{k \leq n}
\overline{F_{f(k)}^k})_{n \in \w}$ is a closed net at $K$ in
$\beta X$, therefore, $(\bigcap_{k \leq n} F_{f(k)}^k)_{n \in
\w}$ is a closed net at $K$ in $X$.
 Conversely,
for each $n \in \w$, let $\mW_n$ be an
open covering for $\beta X \setminus X$. We may suppose that each
$W \in \mW_n$ is open in $\beta X$. By regularity, we can take a
refinement $\mV_n$ of $\mW_n$ such that, for every $x \in \beta X
\setminus X$, there is a $V \in \mV_n$ such that $x \in V \subset
\overline V \subset W_V$ for some $W_V \in \mW_n$. Since $X$ is
 of countable type,  By  \ref{henriksen} we may suppose that
each $\mV_n$ is
countable.  Fix an
enumeration for each $\mV_n = (V_k^n)_{k \in \w}$.  Define $A_k^n
=
\beta X \setminus(\bigcup_{j\le k} \overline{V_j^n})$. Note that
each $K_n =
\bigcap_{k \in \w} \overline{A_k^n}$ is compact and a subset of
$X$. By Lemma \ref{decrescente eh net}, $(\overline{A_k^n})_{k
\in \w}$ is a closed net at $K_n$. Thus, $(\overline{A_k^n} \cap
X)_{k \in \w}$ is a closed net at $K_n$ in $X$. Therefore, there
is $f: \w \Em \w$ such that $K = \bigcap_{n \in
\w}(\overline{A_{f(n)}^n} \cap X)$ is compact and $(\bigcap_{k
\leq f(n)} \overline{A_{f(k)}^k} \cap X)_{n \in \w}$ is a closed
net at $K$. So, by Lemma \ref{nets sobem}, $K = \bigcap_{n \in
\w}\overline{(\overline{A_{f(n)}^n} \cap X)}$. Since $\bigcap_{n
\in \w}\overline{(\overline{A_{f(n)}^n} \cap X)} = \bigcap_{n \in
\w}\overline{A_{f(n)}^n}$ and by the fact that $K \subset X$, it
follows that $\beta X \setminus X  \subset \bigcup_{n \in \w}
\beta X \setminus \overline {A_{f(n)}^n}\subset \bigcup_{n\in \w}
Int(\bigcup_{j\le
f(n)} \overline {V_j^n})\subset \bigcup _{n\in \w} \bigcup_{j\le
f(n)} W_{V^n_j}$. Therefore, letting $\mathcal
U_n=\{W_{V^n_j}:j\le f(n)\}\subset \mW_n$, we see that  the
collection $\bigcup_{n\in \w}
\mathcal U_n$ covers $\beta X\setminus X$, and we are done.
\end{proof}

Property $\mathcal M$  given in the above theorem does not look
very nice and we wonder whether there is  a simpler way to
describe it, at least in some special cases.

   Recall that  a metrizable space is always of countable type.
Moreover,
a metrizable space is complete if and only if it is
$\sigma$-compact at infinity.  Therefore,  we could hope for a
``nicer'' $\mathcal M$ in this case.

\begin{question} What kind of weak completeness characterizes
those metrizable spaces which are Menger at infinity?
\end{question}

\begin{proposition}
  Let $X$ be a Tychonoff space. If $X$ is Menger
at infinity then for every sequence $(K_n)_{n \in \w}$ of compact
sets, there is a sequence $(Q_n)_{n \in \w}$ of compact  sets
such that:
  \begin{enumerate}
  \item each $K_n \subset Q_n$;
  \item each $Q_n$ has a countable base at $X$;
  \item \label{seq}for every sequence $(B_k^n)_{n, k \in \w}$
such that, for every $n \in \w$, $(B_k^n)_{k \in \w}$ is a
decreasing base at $K_n$, then there is a function $f: \w \Em \w$
such that $K = \bigcap_{n \in \w} \overline{B_{f(n)}^n}$ is
compact and $(\bigcap_{k \leq n} \overline{B_{f(k)}^k})_{n \in
\w}$ is a closed net at $K$.
  \end{enumerate}
\end{proposition}

\begin{proof}
  Suppose $X$ is Menger at infinity. Let $(K_n)_{n \in
\w}$ be a sequence of compact sets. Since $X$ is Menger
at infinity, $X$ is Lindel\"of at infinity. Thus, by Proposition
\ref{henriksen}, for each $K_n$, there is a compact $Q_n \supset
K_n$ such that $Q_n$ has a countable base. Now, let
$(B_k^n)_{k, n}$ be as in \ref{seq}. Since each
$Q_n$ is compact and  $X$ is regular, each $(\overline{B_k^n})_{k
\in \w}$ is a decreasing closed net at $Q_n$. Thus, by
Proposition \ref{Menger at infinity}, there is an $f: \w \Em \w$
as we need.
\end{proof}

We end this section presenting a selection principle that at
first glance could be related with the Menger at infinity
property.

\begin{definition}
  We say that a family $\mU$ of open sets of $X$ is an {\bf
almost covering} for $X$ if $X \setminus \bigcup \mU$ is compact.
We call $\mA$ the family of all almost coverings for $X$.
\end{definition}

Note that the property ``being Menger at infinity'' looks like
something as $\sfin(\mA, \mA)$, but for a narrow class of $\mA$.
We will see that the ``narrow'' part is important.

\begin{proposition}\label{AA is Menger}
If $X$ satisfies $\sfin(\mA, \mA)$, then $X$ is Menger.
\end{proposition}

\begin{proof}
  Let $(\mU_n)_{n \in \w}$ be a sequence of coverings of $X$.
By definition, for each $n \in \w$, there is a finite $U_n
\subset
\mU_n$, such that $K = X \setminus \bigcup_{n \in \w} \bigcup
U_n$ is compact. Therefore, there is a finite $W \subset U_n$
such that $K \subset \bigcup W$. Thus, $X = W \cup \bigcup_{n \in
\w}\bigcup U_n$.
\end{proof}

\begin{example}
The space of the irrationals is an example of a space that is
Menger at infinity but does not satisfy $\sfin(\mA, \mA)$ (by the
Proposition \ref{AA is Menger}).
\end{example}

\begin{example}
The  one-point Lindel\"ofication of a discrete space  of
cardinality
$\aleph_1$ is an example of a Menger space which does not
satisfy
$\sfin(\mA, \mA)$.
\end{example}

\begin{example}
  $\w$ is an example of a space that satisfies $\sfin(\mA, \mA)$,
but it is not compact.
\end{example}

\begin{proof}
  Let $(\mV_n)_{n \in \w}$ be a sequence of almost coverings for
$\w$. Therefore, for each $n$, $F_n = \w \setminus \bigcup \mV_n$
is finite. For each $n$, let $V_n \subset \mV_n$ be a finite
subset such that $F_{n + 1} \setminus F_n \subset \bigcup V_n$
and $\min(\w \setminus \bigcup_{k < n} V_k) \in V_n$. Note that
$\w \setminus \bigcup_{n \in \w} \bigcup V_n = F_0$.
\end{proof}

\section {More than Menger at infinity}
One may wonder  whether the hypothesis ``player 2 has a
winning strategy in the Menger game $\gfin(\mO,\mO)$ played on
$\beta
X\setminus X$''  is strong enough to guarantee that $X$ is \v
Cech-complete. It urns out
 this is not the case, as the following example
shows.

\begin{example} \label{ex}
 Take the  usual space of rational numbers $\Q$ and  an
uncountable discrete space $D$.  Let
$Y=\Q\times D \cup \{p\}$ be the one-point Lindel\" ofication of
the space $\Q\times D$ and   then let
$X=\beta Y \setminus Y$. Since $Y$ is nowhere locally compact, we
have $Y=\beta X\setminus X$. $X$ is not \v Cech-complete, since
$Y$ is not $\sigma$-compact, but   player 2 has a winning
strategy
in $\gfin(\mO, \mO)$ played on $\beta X\setminus X$. The latter
assertion easily follows by observing that any open set
containing $p$ leaves out countably many points. \end{example}

Therefore, to ensure the \v Cech-completeness of $X$, we need
to assume something more on the  space (see for instance
Corollary \ref{cech}
below). Moreover,  the first example presented in the
introduction
shows that a metrizable space (actually a subspace of the real
line) can be  Menger at infinity, but not
favorable     for player 2 in the Menger game at infinity (see
again Corollary \ref{cech}).

Recall that a space $X$ is sieve complete \cite{michael} if there
is an indexed
collection of open coverings $\langle \{ U_i: i\in
I_n\}:n<\omega\rangle$ together with  mapps $  \pi_n:I_{n+1}\to
I_n $ such that $U_i=X$ for each $i\in I_0$ and
$U_i=\bigcup\{U_j:j\in \pi_n^{-1}(i)\}$ for all $i\in I_n$.
Moreover, we require that for any sequence of indexes
$\langle i_n:n<\omega\rangle$ satisfying $\pi_n(i_{n+1})=i_n$ if
$\mathcal F$ is a filterbase in $X$ and  $U_{i_n}$ contains an
element of $\mathcal F$ for each $n<\omega$, then $\mathcal F$
has a cluster point.

Every \v Cech-complete space is sieve complete and every sieve
complete space contains a dense \v Cech-complete subspace.
In addition, a paracompact sieve complete space is \v
Cech-complete
and a
sieve complete space is of countable type \cite{topsoe}.

Telg\'arsky  presented a characterization of sieve completeness
in terms of the Menger game played on $\beta X \setminus X$ (note
that in \cite{telgarsky2} the Menger game is called the Hurewicz
game and is denoted by $H(X)$):

\begin{theorem}[Telg\'arsky \cite{telgarsky2}] \label{more} Let
$X$ be a Tychonoff space. $\beta
X\setminus X$ is favorable for player 2 in the Menger game if and
only if $X$ is sieve complete.
\end{theorem}

Since a paracompact sieve-complete space is \v Cech-complete, we
immediately get:
\begin{corollary}  \label{cech} Let $X$ be a paracompact
Tychonoff space. $X$ is \v Cech-complete if and only if  player 2
has a winning strategy in the game $\gfin(\mO, \mO)$ played on
$\beta X\setminus X$.
\end{corollary}
In particular:
\begin{corollary} A metrizable space $X$ is complete if and only
if player 2 has a winning strategy in $\gfin(\mO,\mO)$ played on
$\beta X\setminus X$. \end{corollary}
\begin{corollary} A topological group $G$ is \v Cech-complete if
and only if player 2 has a winning strategy in $\gfin(\mO,\mO)$
played on $\beta G\setminus G$. \end{corollary}
\begin{proof} Every topological group of countable type is
paracompact. \end{proof}
\bibliographystyle{abbrv}

\def\cprime{$'$}

\end{document}